\documentclass{birkjour}
 \newtheorem{thm}{Theorem}[section]

 \theoremstyle{definition}
 
 \theoremstyle{remark}
 
 \newtheorem{ex}{Example}
 \numberwithin{equation}{section}
\newcommand{\cred}[1]{{\color{black}  #1}}
\newcommand{\cblu}[1]{{\color{black}  #1}}

\begin{document}

%-------------------------------------------------------------------------
% editorial commands: to be inserted by the editorial office
%
%\firstpage{1} \volume{228} \Copyrightyear{2004} \DOI{003-0001}
%
%
%\seriesextra{Just an add-on}
%\seriesextraline{This is the Concrete Title of this Book\br H.E. R and S.T.C. W, Eds.}
%
% for journals:
%
%\firstpage{1}
%\issuenumber{1}
%\Volumeandyear{1 (2004)}
%\Copyrightyear{2004}
%\DOI{003-xxxx-y}
%\Signet
%\commby{inhouse}
%\submitted{March 14, 2003}
%\received{March 16, 2000}
%\revised{June 1, 2000}
%\accepted{July 22, 2000}
%
%
%
%---------------------------------------------------------------------------
%Insert here the title, affiliations and abstract:
%

\title[A new block diagonal preconditioner]{A new block diagonal preconditioner for a class of $3\times 3$ block saddle point problems}

%----------Author 1
\author[M. Abdolmaleki]{Maryam  Abdolmaleki}

\address{Faculty of Intelligent Systems Engineering and Data Science, Persian Gulf University, Bushehr, Iran}

\email{maleki.un@gmail.com}

%\thanks{This work was completed with the support of our
%\TeX-pert.}
%----------Author 2
\author[S. Karimi]{Saeed Karimi$^*$}
\address{Faculty of Intelligent Systems Engineering and Data Science, Persian Gulf University, Bushehr, Iran}
\email{karimi@pgu.ac.ir}
%----------Author 3
\author[D. K. Salkuyeh]{Davod Khojasteh Salkuyeh}
\address{Faculty of Mathematical Sciences, University of Guilan, Rasht, Iran\\
Center of Excellence for Mathematical Modelling, Optimization and Combinational Computing
(MMOCC), University of Guilan, Rasht, Iran}
\email{khojasteh@guilan.ac.ir}
%----------classification, keywords, date


\date{January 1, 2020}
%----------additions
%\dedicatory{To my boss}
%%% ----------------------------------------------------------------------

\begin{abstract}
We study the performance of a new block  preconditioner for a class of $3\times3$ block saddle point problems which arise from finite element methods for solving time-dependent Maxwell equations and some other practical problems. We also estimate the lower and upper bounds of eigenvalues of the preconditioned matrix. \cred{Finally, we examine our new preconditioner to accelerate the convergence speed of the GMRES method which shows the effectiveness of the preconditioner.}
\end{abstract}

%%% ----------------------------------------------------------------------
\maketitle
%%% ----------------------------------------------------------------------
%\tableofcontents
 \section{Introduction}\label{Sec1}
Consider the following $3\times 3$ block saddle point problems:
\begin{equation} \label{Sdpr1}
    \mathcal{B}u:\equiv
    \left(
    \begin{array}{ccc}
        A & B^T & 0 \\
        B & 0 & C^T\\
        0 & C & 0\\
    \end{array}	
    \right)
    \left(
    \begin{array}{c}
        x\\
        y\\
        z\\
    \end{array}	
    \right)
    =
    \left(
    \begin{array}{c}
        f\\
        g\\
        h\\
    \end{array}	
    \right),
\end{equation}
\noindent
where  $A\in \mathbb{R}^{n\times n}$ is  a \cred{symmetric positive definite (SPD)} matrix, and  $B\in \mathbb{R}^{m\times n}$ and  $C\in \mathbb{R}^{l\times m}$ have full row rank. In addition, $f\in \mathbb{R}^{n}$, $g\in \mathbb{R}^{m}$ and $h\in \mathbb{R}^{l}$ are given vectors, and  $u$ is an unknown vector which is to be determined. Furthermore, we presume that the matrices $A$ and $B$ are large and sparse. It is not difficult to check that under the above conditions the coefficient matrix $\mathcal{B}$ is nonsingular and as a result system \eqref{Sdpr1} has a unique solution \cite{last}. In this work, we focus on preconditioned Krylov-subspace methods, especially, \cred{the preconditioned  GMRES method} (see \cite{fgmres1,fgmres2}). 	

Linear systems of the form \eqref{Sdpr1} appear in a variety of scientific and engineering problems, for instance, full discrete finite element methods for solving the time-dependent Maxwell equations with discontinuous coefficients  \cite{Assous,Chen,Ciarlet}, the following quadratic program \cite{Han}:
\[
\min\{\frac{1}{2} x^T Ax+r^T x+q^T y\},\quad
\]
\[
s.t.\quad  Bx+c^T y=b,\quad x\in \mathbb{R}^n,\quad y\in \mathbb{R}^l,
\]
where $r\in \mathbb{R}^n$ and  $q\in \mathbb{R}^l$
are given vectors, Picard iteration schemes for variational formulation of the stationary incompressible magnetohydrodynamics system \cite{Hu} and least squares problems \cite{LS}.

Linear systems of $2\times2$ block form,

\begin{equation} \label{two by two}
    \left(
    \begin{array}{cc}
        \tilde{A} & {\tilde{B}}^T \\
        \tilde{B} & -\tilde{C} \\
    \end{array}	
    \right)
    \left(
    \begin{array}{c}
        x\\
        y\\
    \end{array}	
    \right)
    =
    \left(
    \begin{array}{c}
        f\\
        g\\
    \end{array}	
    \right),
\end{equation}
known as traditional saddle point problems, have  been extensively studied for decades, where  $\tilde{A}$ and $\tilde{C}$ are positive and  positive semi-definite matrices respectively, and $\tilde{B
}$ is  a full row rank matrix. Constructing preconditioners to improve the convergence speed of Krylov-subspace methods for solving \eqref{two by two} has produced a considerable amount of literature, e.g., shift-splitting preconditioners \cite{shift3,shift1,shift2,shift4}, block triangular preconditioners \cite{bt1,bt3,bt2,our}, inexact constraint preconditioners\cite{cp2,cp1,cp3} and so on. It is obvious that the $3\times 3$ block linear system \eqref{Sdpr1} can be seen as a special case of the  traditional $2\times 2$ form \eqref{two by two} using the following partitioning strategies,
\begin{equation}\label{part1}
    \begin{pmatrix}
        A & B^T & \vdots & 0\\
        B & 0 & \vdots & C^T \\
        \cdots & \cdots & \vdots & \cdots\\
        0 & C & \vdots & 0
    \end{pmatrix},
    \qquad
\end{equation}
or
\begin{equation}\label{part2}
    \begin{pmatrix}
        A & \vdots & B^T & 0\\
        \cdots & \cdots &\cdots & \cdots \\
        B & \vdots & 0 & C^T \\
        0 & \vdots & C & 0
    \end{pmatrix}.
\end{equation}
However, \cred{most of the methods for the latter matrices cannot be directly applied to \eqref{Sdpr1}}. This is because the properties of the sub-matrix in \eqref{part1} and \eqref{part2} are different from the traditional form \eqref{two by two}. Indeed, the $(1,1)$ leading block matrix in \eqref{part1}  is symmetric indefinite and a standard saddle point matrix. Thus, the linear system with the coefficient matrix \eqref{part1} can be considered as a double saddle point problem studied recently \cite{Double1,Double2}. In addition, the $(2,2)$ block in \eqref{part2} is a symmetric matrix and $(1,2)$ block is rank deficient. Therefore, it is extremely  interesting to find new efficient preconditioners for the $3\times 3$ block saddle point problems \eqref{Sdpr1}.

Some preconditioners have been studied to accelerate convergence rate of the Krylov-subspace methods for solving this class of $3\times 3$ systems,  for example \cite{Ex1-2,3-2,3-1}. Recently, Huang and Ma \cite{Ex1-1} proposed two block diagonal (BD) preconditioners,
\begin{equation}\label{BD2}
    \mathcal{P}_{BD1}=\begin{pmatrix}
        A & 0 & 0\\
        0 & S & 0 \\
        0 & 0 & CS^{-1} C^T
    \end{pmatrix}
    \qquad \text{and} \qquad \mathcal{P}_{BD2}=
    \begin{pmatrix}
        \hat{A} & 0 & 0\\
        0 & \hat{S} & 0 \\
        0 & 0 & CS^{-1} C^T
    \end{pmatrix},
\end{equation}
for solving \eqref{Sdpr1} in which $S=BA^{-1} B^T$, $\hat{A}$ and $\hat{S}$ are SPD approximations of $A$ and $S$, respectively. They also derive all the eigenpairs of preconditioned matrix. Subsequently, Xie and Li \cite{last} proposed three new preconditioners for solving the linear system \eqref{Sdpr1} that can be seen as follows
\begin{equation}
\begin{aligned}\label{p2}
    \mathcal{P}_{1}=\begin{pmatrix}
        A & 0 & 0\\
        B & -S & C^T \\
        0 & 0 & -CS^{-1}C^T
    \end{pmatrix},\quad &
    \mathcal{P}_{2}=\begin{pmatrix}
        A & 0 & 0\\
        B & -S & C^T \\
        0 & 0 & CS^{-1}C^T
    \end{pmatrix}, \\   
    \mathcal{P}_{3}=\begin{pmatrix}
        A & B^T & 0\\
        B & -S & C^T \\
        0 & 0 & -CS^{-1}C^T
    \end{pmatrix},\quad
\end{aligned}
\end{equation}
where $S=BA^{-1}B^T$. They analyzed spectral properties of corresponding preconditioned matrices and showed that the proposed preconditioners significantly accelerate the convergence rate of GMRES method. \cred{ However, when the inner systems are solved inexactly, the elapsed CPU time is increased drastically and often give unacceptable solutions.}

Here, we consider the following equivalent form of \eqref{Sdpr1}:

\begin{equation}\label{equ}
    {\mathcal{A}}u:\equiv
    \left(
    \begin{array}{ccc}
        A & B^T & 0 \\
        -B & 0 & -C^T\\
        0 & C & 0\\
    \end{array}	
    \right)
    \left(
    \begin{array}{c}
        x\\
        y\\
        z\\
    \end{array}	
    \right)
    =
    \left(
    \begin{array}{c}
        f\\
        -g\\
        h\\
    \end{array}	
    \right)
    \equiv b.
\end{equation}
Although the coefficient matrix of the system \eqref{equ} is not symmetric, it has some desirable properties. For instance, the matrix $\mathcal{A}$ is positive semi-definite, i.e.,  $\mathcal{A}+{\mathcal{A}}^T$ is symmetric positive semi-definite. This is a signification for the GMRES method. In fact, the restarted version of GMRES($m$) converges for all $m\geq1$. 
\cred{It is noticeable that setting up the exact preconditioners mentioned above are very time-consuming and the inexact preconditioners need approximations of $A$ and $S$.} 
In this paper, we establish a new block preconditioner  for solving the linear system \eqref{equ} which is easy to implement 
and has much better computing efficiently than the preconditioners studied recently.

It  is noteworthy that, few analytical results on spectral bounds are  available for a $3\times 3$ block matrix of the form  $\mathcal{A}$. In fact, contrary to the case for $2\times 2$ block matrix, we get a cubic equation from the eigen-system of a $3\times 3$ block matrix. Therefore, estimating  the bounds of eigenvalues brings some difficulties.
Consider the following  monic polynomial of degree $n\geq2$:
\begin{equation}
    p(z)=z^n+\sum_{k=0}^{n-1} a_k z^k, \qquad a_i\in \mathbb{C}, \quad i=0,1,\ldots,n-1.
\end{equation}
\cred{Finding approximately the roots of $p(z)$ through simple operations with its coefficients has led to publish a plenty of studies, for example, see the comprehensive surveys \cite{Marden,Sendov} and references therein for further details}. We will derive the spectral bounds of corresponding preconditioned system by using one of the classical and sharp bounds that has been used to obtain simple lower and upper bounds on the absolute value of the roots of $p(z)$.

This paper is divided into three sections, schemed as follows. In Section \ref{sec2}, the new preconditioner is presented and the clustering properties of corresponding preconditioned system are discussed. Numerical results are given in Section \ref{sec3} to demonstrate the effectiveness of the new preconditioner. The paper is ended by some concluding remarks in Section \ref{sec4}.

We end this section with an introduction of some notation that will be used in the subsequent sections. The symbol $x^*$ is used for the conjugate transpose of the vector $x$. For any square matrix $A$ with real eigenvalues, the minimum and maximum eigenvalues of $A$ are indicated by $\lambda_{\min}(A)$ and $\lambda_{\max}(A)$, respectively. The norm $\|.\|$ indicates the \cred{Euclidean norm}. \cred{Moreover, we use \textsc{Matlab} notation  $(x;y;z)$ to denote the vector $(x^T,y^T,z^T)^T$}.

\section{The new block diagonal preconditioner} \label{sec2}
We propose the following block preconditioner for solving the linear system \eqref{equ}
\begin{equation}\label{prec}
    \mathcal{M}=\begin{pmatrix}
        A & 0 & 0 \\
        0 & \alpha I+\beta BB^T & 0 \\
        0 & 0 & \alpha I+\beta CC^T
    \end{pmatrix},
\end{equation}
\noindent
where $\alpha, \beta >0$. \cred{The main advantage of the preconditioner $\mathcal{M}$ over the preconditioners mentioned in Section \ref{Sec1} is that it is free of the Schur complement matrix $S=BA^{-1} B^T$ and easy to implement.}  Concerning the clustering properties of the eigenvalues of the preconditioned matrix $\mathcal{M}^{-1}\mathcal{A}$, we have the following theorems.

\begin{thm}\label{bound}
    \cite{Bound} Let $p(z)=z^n+\sum_{k=0}^{n-1} a_k z^k$ be a monic polynomials with complex coefficients and $\lambda$ be any root of $p(z)$. Then $|\lambda|$ satisfies the following inequalities
    \begin{itemize}
        \item Cauchy's lower and upper bounds
        \small{
       \begin{align*}
        \frac{\left|a_0\right| }{\max\lbrace 1,\left| a_0 \right|+\left| a_1 \right|,\left| a_0 \right|+\left| a_2 \right|,\ldots,\left| a_0 \right|+\left| a_{n-1} \right|\rbrace  } & \\
        \leq \left|\lambda \right| \leq& \max \lbrace \left| a_0 \right|,1+\left| a_1\right|,\ldots,1+\left| a_{n-1} \right| \rbrace.
         \end{align*}}
        \item Montel's lower and upper bounds
        \[
        \frac{\left|a_0\right|}{\max\lbrace \left| a_0 \right|, 1+\left| a_1 \right|+\left| a_2 \right|+\cdots+\left|a_{n-1}  \right|\rbrace} \leq \left|\lambda \right|  \leq \max \lbrace 1, \left| a_0\right|+\left| a_1\right|+\cdots +\left| a_{n-1} \right| \rbrace.
        \]
        \item Carmichael-Mason's lower and upper bounds
        \[
        \frac{\left|a_0\right|}{\sqrt{ 1+\left| a_0 \right| ^2 +\left| a_1 \right| ^2 +\cdots+\left|a_{n-1}  \right|^2 }} \leq \left|\lambda \right|   \leq \sqrt{ 1+\left| a_0 \right| ^2 +\left| a_1 \right| ^2 +\cdots+\left|a_{n-1}  \right|^2 }.
        \]
        \item Frobenius' lower and upper bounds
        \[
        \frac{\left|a_0\right|}{\sqrt{ 1+(n-1)\left| a_0 \right| ^2 +\left| a_1 \right| ^2 +\cdots+\left|a_{n-1}  \right|^2 }} \leq \left|\lambda \right|   \leq \sqrt{ (n-1)+\left| a_0 \right| ^2 +\cdots+\left|a_{n-1}  \right|^2 }.
        \]

    \end{itemize}

\end{thm}
It is noted that, Cauchy's bounds are essentially the most sharpest ones in Theorem \ref{bound} and have been used for testing the sharpness of other new bounds. Hence, we will utilize these bounds to locate the  eigenvalues of corresponding preconditioned matrix.

\begin{thm}\label{th2}
    Suppose that $A\in\mathbb{R}^{n\times n}$ is SPD, and $B\in\mathbb{R}^{m\times n}$ and $C\in\mathbb{R}^{l\times m}$ have full row rank. Then the preconditioned matrix $\mathcal{M}^{-1}\mathcal{A}$ has an eigenvalue $1$ of \cred{algebraic multiplicity} $n-m$, i.e. $\lambda_i ^{(1)}=1$, $i=1,2,\ldots,n-m$, and the corresponding eigenvectors are of the form $(x_i;0;0),$ where $\{x_i\}_{i=1} ^{n-m}$  is a basis for the null space of $B$. Moreover, for the remaining eigenvalues we have the following statements:
    \begin{itemize}
        \item If $p+q>1$, $ \qquad \frac{p}{1+2p+q} \leq \left| \lambda - 1 \right| <  2+p+q$,
        \item If $p+q\leq 1$, $\qquad \quad \frac{p}{2+p}\leq \left| \lambda - 1 \right| \leq  3$,
    \end{itemize}
    where
    \[
    p=\frac{y^* BA^{-1}B^Ty}{y^*(\alpha I+\beta BB^T)y} \qquad \text{and} \qquad
    q=\frac{y^* C^T(\alpha I+\beta CC^T)^{-1}Cy}{y^* (\alpha I+\beta BB^T)y},
    \]
    for some $0\neq y\in\mathbb{R}^m$ and $\alpha,\beta>0$.
\end{thm}
\begin{proof}
    Let $(\lambda,(x;y;z))$ be an eigenpair of $\mathcal{M}^{-1}\mathcal{A}$. Then, we have
    \small{
    \[
    \left(
    \begin{array}{ccc}
        A & B^T & 0 \\
        -B & 0 & -C^T\\
        0 & C & 0\\
    \end{array}	
    \right)
    \left(
    \begin{array}{c}
        x\\
        y\\
        z\\
    \end{array}	
    \right)
    =
    \lambda
    \left(
    \begin{array}{ccc}
        A & 0 & 0 \\
        0 & \alpha I+\beta BB^T& 0\\
        0 & 0 & \alpha I+\beta CC^T\\
    \end{array}	
    \right)
    \left(
    \begin{array}{c}
        x\\
        y\\
        z\\
    \end{array}	
    \right),
    \]}
    which can be rewritten as

    \begin{subequations}
        \begin{empheq}[left=\empheqlbrace]{align}
            &Ax+B^Ty=\lambda Ax, \label{a} \\
            &Bx+C^Tz=-\lambda (\alpha I+\beta BB^T)y, \label{b} \\
            &Cy=\lambda (\alpha I+\beta CC^T)z. \label{c}
        \end{empheq}
    \end{subequations}
    \noindent
    If $x=0 $, then from \eqref{a} we obtain $B^Ty=0$, which shows that $y=0$. This along with \eqref{c} yield that $z=0$, which is a contradiction with the fact that $(x;y;z)$ is an eigenvector. Next, we complete the proof in the following cases.

    Firstly, we consider the case that $y=0$. Since $\lambda \neq 0$ and $\alpha I+\beta CC^T$ is nonsingular, we deduce that $z=0$. Then \eqref{a} and \eqref{b} are reduced to
    \[
    Ax=\lambda Ax, \qquad Bx=0,
    \]
    respectively. This shows that $\lambda=1$ and  the corresponding eigenvectors are of the form $(x;0;0)$ with $x\in null(B).$

    Next, we consider the case that $y\neq 0$. Then from \eqref{a}, we know that $\lambda \neq 1$. Noticing that both of matrices $\mathcal{A}$ and   $\alpha I+\beta CC^T$  are nonsingular, it follows from \eqref{a} and \eqref{c} that
    \[
    x=\frac{1}{\lambda-1}A^{-1}B^Ty, \qquad z=\frac{1}{\lambda}( \alpha I+\beta CC^T)^{-1}Cy.
    \]
    Substituting the preceding equalities into \eqref{b}, gives 
    \[
    \frac{1}{\lambda -1}BA^{-1}B^Ty+\frac{1}{\lambda}C^T
    (\alpha I+\beta CC^T)^{-1}Cy=-\lambda (\alpha I+\beta BB^T)y.
    \]
    By premultiplying the above equation by $\lambda (\lambda -1)$ and $y^*$, we obtain the following cubic equation:
    \begin{equation} \label{cubic}
        \lambda^3 -\lambda^2 +(p+q)\lambda -q=0,
    \end{equation}
    where
    \begin{equation}\label{p&q}
        p=\frac{y^* BA^{-1}B^Ty}{y^*(\alpha I+\beta BB^T)y} \qquad \text{and} \qquad
        q=\frac{y^* C^T(\alpha I+\beta CC^T)^{-1}Cy}{y^* (\alpha I+\beta BB^T)y}.
    \end{equation}
    To obtain the clustering properties of the eigenvalues around the point $(1,0)$ we set $\mu=\lambda -1$. By substituting $\mu$ in \eqref{cubic} we get the following cubic equation
    \begin{equation}\label{cubic2}
        \mu ^3 + 2\mu ^2 +(1+p+q)\mu +p=0,
    \end{equation}
    and find upper and lower bounds for $\mu$. From  Theorem \ref{bound}, we deduce that the Cauchy's lower and upper bounds for $\mu$ are
    \begin{subequations}
        \begin{empheq}[left=\empheqlbrace]{align}
            \frac{p}{1+2p+q} \leq \left| \mu \right| < & 2+p+q, & \text{if} \quad p+q>1, \label{CAu1}\\
            \frac{p}{2+p}\leq \left| \mu \right| \leq & 3,  & \text{if} \quad p+q\leq 1,\label{CAU2}
        \end{empheq}
    \end{subequations}
which completes the proof. 
\end{proof}

As we see in Theorem \ref{th2}, the lower and upper bounds of some eigenvalues of the preconditioned matrix $\mathcal{M}^{-1}\mathcal{A}$  depend on the parameters $\alpha$ and $\beta$. Using \eqref{p&q} and  the Schur decomposition of matrices $A$ and $\alpha I+\beta CC^T$, and straightforward computations the following upper and lower bounds for $p$ and $q$ are  obtained as
\[
\frac{\lambda_{\min} (A^{-1}) \lambda_{\min} (BB^T)}{\alpha +\beta\lambda_{\max} (BB^T)}\leq p \leq \frac{\lambda_{\max} (A^{-1}) \lambda_{\max} (BB^T)}{\alpha +\beta\lambda_{\min} (BB^T)},
\]
and
\[
\frac{{\lambda_{\min} (C^TC)} (\alpha+\beta \lambda_{\max} (CC^T))^{-1}}{\alpha +\beta\lambda_{\max} (BB^T)}\leq q \leq \frac{\lambda_{\max} (CC^T) (\alpha +\beta \lambda_{\min} (CC^T))^{-1}}{\alpha +\beta \lambda_{\min} (BB^T)}.
\]
\noindent
Hence, we conclude that
\begin{multline}\label{lam}
    \frac{\lambda_{\min} (A^{-1}) \lambda_{\min} (BB^T) +{\lambda_{\min} (C^TC)}(\alpha +\beta\lambda_{\max} (CC^T))^{-1}}{\alpha +\beta\lambda_{\max}(BB^T)}
    \\ \leq p+q \leq
    \frac{\lambda_{\max} (A^{-1}) \lambda_{\max} (BB^T) +\lambda_{\max} (CC^T)(\alpha +\beta\lambda_{\min} (CC^T))^{-1}}{\alpha +\beta\lambda_{\min}(BB^T)}.
\end{multline}
\cred{Now, we consider the following two cases:}

\noindent
\textbf{case I:} \quad $p+q>1$: \quad From \eqref{lam} we deduce that all we need is to seek the parameters $\alpha$ and $\beta$ such that
\[
1<\frac{\lambda_{\min} (A^{-1}) \lambda_{\min} (BB^T) +{\lambda_{\min} (C^TC)}(\alpha +\beta\lambda_{\max} (CC^T))^{-1}}{\alpha +\beta\lambda_{\max}(BB^T)}.
\]
For the sake of the simplicity, let
\[
\eta=\lambda_{\max}(BB^T) + \lambda_{\max}(CC^T), \qquad \theta =\lambda_{\max}(BB^T) \lambda_{\max}(CC^T),
\]
\[
\kappa=\lambda_{\min} (A^{-1}) \lambda_{\min} (BB^T),\qquad \gamma=\lambda_{\min} (A^{-1}) \lambda_{\min} (BB^T) \lambda_{\max} (CC^T),
\]
\[
{\zeta = \lambda_{\min}(C^TC)} .
\]
Therefore, we have the following inequality
\[
\alpha^2 + (\eta \beta -\kappa)\alpha+\theta \beta^2 -\gamma \beta - \zeta <0.
\]
It is clear that matrices $CC^T$ and $BB^T$ are both SPD under our assumptions. Hence, we infer that $\eta \neq 0$. Therefore, we can set
\[
\beta = \frac{\kappa}{\eta}\equiv \frac{\lambda_{\min} (A^{-1}) \lambda_{\min} (BB^T)}{\lambda_{\max} (BB^T)+ \lambda_{\max} (CC^T)}
\]
and derive that
\[
0<\alpha<\sqrt{\frac{\gamma \kappa}{\eta}-\frac{\theta \kappa ^2}{\eta ^2}+\zeta}.
\]
Note that $\theta\kappa<\eta\gamma$, and then it follows that $\frac{\kappa}{\eta}(\gamma-\frac{\theta\kappa}{\eta})>0$. Therefore, the above inequality can be easily deduced.

\noindent
\textbf{case II:} \quad $p+q\leq1$: \quad In this case, it is enough to find $\alpha$ and $\beta$ such that
\[
\frac{\lambda_{\max} (A^{-1}) \lambda_{\max} (BB^T) +\lambda_{\max} (CC^T)(\alpha +\beta\lambda_{\min} (CC^T))^{-1}}{\alpha +\beta\lambda_{\min}(BB^T)} \leq 1.
\]
Hence, we have
\[
\alpha ^2 +(\eta ^\prime \beta -\kappa ^ \prime)\alpha+\beta ^2 \theta ^\prime -\beta \gamma ^\prime -\zeta^ \prime \geq 0,
\]
where
\[
\eta ^\prime=\lambda_{\min} (BB^T)+\lambda_{\min} (CC^T), \quad \theta^\prime=\lambda_{\min} (CC^T)\lambda_{\min} (BB^T),
\]
\[
\kappa^\prime=\lambda_{\max} (A^{-1})\lambda_{\max} (BB^T), \quad \gamma^\prime = \lambda_{\max} (A^{-1}) \lambda_{\max} (BB^T)\lambda_{\min} (CC^T),
\]
\[
\zeta^\prime = \lambda_{\max} (CC^T).
\]
Similar to the argument in the previous case we set
\[
\beta = \frac{\kappa^\prime}{\eta^\prime}\equiv \frac{\lambda_{\max} (A^{-1}) \lambda_{\max} (BB^T)}{\lambda_{\min} (BB^T)+\lambda_{\min} (CC^T)},
\]
and finally conclude that
\[
\alpha \geq \sqrt{\frac{\gamma^\prime \kappa^\prime}{\eta^\prime} -\frac{\theta^\prime \kappa^{\prime^{2}}}{\eta^{\prime^{2}}}+\zeta ^\prime}.
\]
\cred{According to the  above discussion we state the following theorem that gives conditions on $\alpha$ and $\beta$ under which  $p+q>1$ or $p+q\leq 1$.}
\begin{thm}\label{th3}
We have the following cases:

    \noindent
    \textbf{case I:} If
    \[
    \beta=\frac{\kappa}{\eta}\qquad
    \text{and}\qquad
    0<\alpha<\sqrt{\frac{\gamma \kappa}{\eta}-\frac{\theta \kappa ^2}{\eta ^2}+\zeta},
    \]
    then, $p+q>1$.

    \noindent
    \textbf{case II:} If
    \[
    \beta = \frac{\kappa^\prime}{\eta^\prime} \qquad \text{and} \qquad  \alpha \geq \sqrt{\frac{\gamma^\prime \kappa^\prime}{\eta^\prime} -\frac{\theta^\prime \kappa^{\prime^{2}}}{\eta^{\prime^{2}}}+\zeta ^\prime},
    \]
    then, $p+q \leq 1$.
\end{thm}

\cred{In each iteration of a Krylov subspace method for solving the preconditioned system $\mathcal{M}^{-1} {\mathcal{A}}u=\mathcal{M}^{-1}b$ we need to compute a vector of the form $(z_1;z_2;z_3)=\mathcal{M}^{-1}(r_1;r_2;r_3)$ where $r_1\in\mathbb{R}^{n}$, $r_2\in\mathbb{R}^{m}$ and $r_3\in\mathbb{R}^{\ell}$. To do so, it is enough to solve the system $\mathcal{M}(z_1;z_2;z_3)=(r_1;r_2;r_3)$ for $(z_1;z_2;z_3)$. Since, $\mathcal{M}$ is a block diagonal matrix, solution of the system is reduced to the solution of three systems with the coefficient matrices $A$, $\alpha I+\beta BB^T$ and $\alpha I+\beta CC^T$, which are all SPD.  By summarizing the above notes we can state  Algorithm \ref{alg}.

\begin{algorithm}\label{Alg1}
    \caption{Computation of $(z_1;z_2;z_3)=\mathcal{M}^{-1}(r_1;r_2;r_3)$}\label{alg}
    \begin{enumerate}
        \item Solve $\mathcal{M}_1 z_1 \equiv Az_1=r_1$ for $z_1$.
        \item Solve $\mathcal{M}_2 z_2 \equiv (\alpha I+\beta BB^T)z_2=r_2$ for $z_2$.
        \item Solve $\mathcal{M}_ 3 z_3 \equiv (\alpha I+\beta CC^T)z_3=r_3$ for $z_3$.
    \end{enumerate}
\end{algorithm}

Since the coefficient matrices of the subsystems in Algorithm \ref{alg} are SPD, they can be solved exactly using  either the Cholesky factorization  or inexactly using the conjugate gradient (CG) iteration method. It is noted that, the shift matrix $\alpha I$ in the steps 3 and 4 of Algorithm \ref{alg} increases the convergence speed of the CG method considerably.}

\section{Numerical experiments} \label{sec3}
\cblu{
In this section, we experimentally compare the effectiveness our proposed preconditioner with two preconditioners $\mathcal{P}_{BD2}$ and  $\mathcal{P}_{2}$ defined, respectively, in Eqs. \eqref{BD2} and \eqref{p2} for solving the saddle point linear system \eqref{equ}.
 All the numerical experiments were computed in double precision using some \textsc{Matlab} codes on a Laptop with Intel Core i7 CPU 2.9 GHz, 16GB RAM.

We apply the preconditioners to accelerate the convergence of the GMRES  (flexible and full versions) method \cite{fgmres1,fgmres2}.
In the implementation of the preconditioners $\mathcal{M}$, $\mathcal{P}_{BD2}$ and $\mathcal{P}_{2}$ within the GMRES method three subsystems with SPD coefficient matrices need to be solved. When the subsystems are solved inexactly  we apply the flexible GMRES (FGMRES) method to solve the preconditioned system and if the subsystems are solved exactly we employ the full version of GMRES (Full-GMRES) method. 
 
In the subsequent presented numerical results, the right-hand side vector $b$ is set to $b=\mathcal{A} u^{\ast}$, where $u^{\ast}$ is a vector of all ones. We use a null vector as an initial guess for the GMRES method and the iteration is stopped once the relative residual 2-norm satisfies
\begin{equation}\label{Rk}
    R_k=\frac{\|b-\mathcal{A} u^{(k)}\|_2}{\|b\|_2}<10^{-6},
\end{equation}
where $u^{(k)}$ is the $k$th computed approximate solution. The maximum number of  iterations is set to be  $maxit=1000$. 
To show the accuracy of the methods we report the values
\begin{equation}\label{Ek}
E_k=\frac{\|u^{(k)}-u^*\|_2}{\|u^*\|_2}.
\end{equation}

 In the implementation of the preconditioners $\mathcal{M}$, $\mathcal{P}_{BD2}$ and $\mathcal{P}_{2}$ three subsystems with SPD coefficient matrices need to be solved in each iteration of the GMRES method. These systems are solved using the Cholesky factorization of matrices, when we apply the Full-GMRES method for solving the preconditioned system. On the other hand, the subsystems are solved using the CG method, when we employ the FGMRES method for solving the preconditioned system. In this case, for solving the subsystems the initial guess is set to be a zero vector and the iterations is stopped as soon as the residual norm is reduced by a factor of $10^3$ or the number of iterations exceeds $500$.

 Numerical results are presented in the tables. To show the effectiveness of the preconditioners  we  also report the numerical results of  Full-GMRES  without any preconditioner.}

%==================================EXAMPLE 1 ================================================================
\begin{ex}\label{ex1} \cite{Ex1-1,Ex1-2}
    Consider the saddle point problem \eqref{equ} for which
    \[
    A=\begin{pmatrix}
        I\otimes T +T\otimes I & 0\\
        0 & I\otimes T +T\otimes I
    \end{pmatrix} \in \mathbb{R}^{2p^2\times 2p^2},
    \]
    \[
    B=\begin{pmatrix}
        I\otimes F & F\otimes I
    \end{pmatrix}\in \mathbb{R}^{p^2\times 2p^2}, \quad C=E\otimes F \in \mathbb{R}^{p^2\times p^2}
    \]
    and
    \[
    T=\frac{1}{h^2} \text{tridiag} (-1,2,-1)\in \mathbb{R}^{p\times p}, \quad F=\frac{1}{h}\text{tridiag}(0,1,-1) \in \mathbb{R}^{p\times p}
    \]
    \[
    E=\text{diag}(1,p+1,\ldots,p^2-p+1),
    \]
    with $\otimes$ being the Kroneker product symbol and $h={1}/{(p+1)}$ the discretization meshsize. Here, the total number of unknowns is $4p^2$ and $m=l$.

Numerical results are presented in Table \ref{tab1} for different values of $p$. This table shows the number of iterations (denoted by ``Iters") and CPU time (denoted by ``CPU") of the FGMRES without preconditioning, and with the preconditioners $\mathcal{M}$, $\mathcal{P}_{2}$ and $\mathcal{P}_{BD2}$. We set $\hat{S}=B{diag(A)}^{-1} B^T$. Moreover, we provide the elapsed CPU time (denoted by ``Prec.CPU") for setting up the preconditioners $\mathcal{P}_{2}$ and $\mathcal{P}_{BD2}$ and the total CPU time (denoted by ``Total.CPU"). Furthermore, for the preconditioner $\mathcal{M}$  we set $\alpha=10^{-3}$ and $\beta=1$. In the all tables, the symbols $\dagger$ and $\ddag$ are used to indicate that the method has not converged in 1000 seconds and $maxit$, respectively. As we see, the preconditioner $\mathcal{M}$ outperforms the other examined preconditioners from the iteration counts, CPU time and the accuracy of computed solution points of view. It should be mentioned that for $p=64,128,256$  the FGMRES method without preconditioning, and the FGMRES with the preconditioners $\mathcal{P}_{2}$ and $\mathcal{P}_{BD2}$ fail to converge in 1000 iterations. 
Therefore, our preconditioner is more effective and practical than the preconditioners $\mathcal{P}_{2}$ and $\mathcal{P}_{BD2}$ 
for solving saddle point problems of the form \eqref{equ}. 

Numerical results of the Full-GMRES method in conjunction with the three preconditioners $\mathcal{M}$, $\mathcal{P}_{2}$ and $\mathcal{P}_{BD2}$ are  shown in Table \ref{tab2}.
% Also, for comparing the results of the  preconditioned GMRES and FGMRES methods with the proposed preconditioner, we have added the results of the preconditioned FGMRES with the preconditioner $\mathcal{M}$ (denoted by $F(\mathcal{M})$) in this table. 
We observe that the preconditioner $\mathcal{M}$ has provided quite suitable results and this results are in good agreement with what we claimed above for the preconditioned FGMRES method. In addition, when $p$ is large the preconditioned GMRES method fails to converge for the preconditioners   $\mathcal{P}_{2}$ and $\mathcal{P}_{BD2}$  in 1000 seconds. However, the FGMRES method with the new preconditioner requires less CPU time. It is noted that we could not set up the preconditioners $\mathcal{P}_{2}$ and $\mathcal{P}_{BD2}$  for $p=512$, because of memory limitation.
\end{ex}

%==================================Table 1 ================================================================

\begin{table}
    \small{
    \caption{Numerical results of FGMRES for  Example \ref{ex1}. } \label{tab1}
    \centering
    \begin{tabular}{|c||c|c|c|c|c|}\hline
        $p$ & 16 & 32 & 64 & 128 & 256 \\ \hline
        \rowcolor{lightgray}
        \multicolumn{6}{|c|}{No Preconditioning}\\ \hline
        Iters & 425 & 949 & $\ddag$ & $\ddag$ & $\ddag$ \\
        CPU & 0.24 & 5.02 & $\dagger$ & $\dagger$ & $\dagger$ \\
        $R_k$ & $8.6\times10^{-7}$ & $9.9\times10^{-7}$ & - &  - &  - \\
        $E_k$ & $2.6\times10^{-6}$ & $2.4\times10^{-5}$ & - &  - & - \\ \hline
        \rowcolor{lightgray}
        \multicolumn{6}{|c|}{\textbf{$\mathcal{M}$}}\\
        %\rowcolor{lightgray}\multicolumn{6}{|c|}{\textbf{$\alpha =10^{-3}$ , $\beta = 1$}}\\
        \hline
        Iters & 109 & 80 & 65 & 71 & 78 \\
        CPU & 0.15 & 0.25 & 1.17 & 6.66 & 36.99 \\
        $R_k$ &  $7.0\times10^{-7}$ &  $8.4\times10^{-7}$ &  $7.1\times10^{-7}$ &  $9.0\times10^{-7}$ &  $1.0\times10^{-6}$ \\
        $E_k$ &  $3.0\times10^{-7}$ &  $6.0\times10^{-7}$ &  $1.5\times10^{-6}$ &  $8.3\times10^{-6}$ &  $1.4\times10^{-5}$ \\ \hline
        \rowcolor{lightgray}
        \multicolumn{6}{|c|}{\textbf{$\mathcal{P}_{2}$}}\\ \hline
        Iters & 124 & 616 &  $\ddag$ & $-$  & $-$  \\
        Prec.CPU & 0.0006 & 0.0015 & 0.0082  &  $-$ & $-$  \\
        CPU & 0.81 & 25.53 &  288.93 &   $\dagger$ & $\dagger$ \\
        Total.CPU & 0.81 & 25.53 & 288.93  &  $-$  & $-$  \\
        $R_k$ &  $9.9\times10^{-7}$ & $8.3\times10^{-7}$ & $1.8\times10^{-1}$ &  $-$  & $-$  \\
        $E_k$ & $1.7\times10^{-6}$ & $1.7\times10^{-6}$ & $6.8\times10^{-1}$  &  $-$ &  $-$ \\ \hline
        \rowcolor{lightgray}
        \multicolumn{6}{|c|}{\textbf{$\mathcal{P}_{BD2}$}}\\ \hline
        Iters & 163 & 785 & $\ddag$  & $-$  & $-$  \\
        Prec.CPU & 0.0006 & 0.0015 & 0.0082   & $-$  &   $-$ \\
        CPU & 0.99 & 34.66 & 286.02  & $\dagger$ &  $\dagger$  \\
        Total.CPU & 0.99 & 34.36 &  286.03 &  $-$ &  $-$  \\
        $R_k$ &  $8.1\times10^{-7}$ & $7.1\times10^{-7}$ & $1.2\times10^{-1}$  &  $-$ &  $-$  \\
        $E_k$ & $1.0\times10^{-6}$ & $1.1\times10^{-6}$ & $6.5\times10^{-1}$ & $-$  &   $-$ \\ \hline
    \end{tabular}}
\end{table}

%==================================Table 2 ===============================================================
\small{
\begin{table}
     {    \caption{Numerical results of  Full-GMRES  for  Example \ref{ex1}.}\label{tab2}
    \centering
    \begin{tabular}{|c||c|c|c|c|c|}\hline
        $p$ & 16 & 32 & 64 & 128 &256  \\ \hline
        %\rowcolor{lightgray}
       \rowcolor{lightgray} \multicolumn{6}{|c|}{No Preconditioning}\\ \hline
        Iters & 425 & 949 & $\ddag$ & $\ddag$ & $\ddag$  \\
        CPU & 0.21 & 5.35 & 14.27 & 45.34 & 182.99  \\
        $R_k$ & $8.6\times10^{-7}$ & $9.9\times10^{-7}$ & $2.7\times10^{-3}$  & $6.7\times10^{-3}$ & $4.9\times10^{-2}$ \\
        $E_k$ & $2.6\times10^{-6}$ & $2.4\times10^{-5}$ & $1.8\times 10^{-1}$   &  $5.5\times10^{-1}$ & $7.8\times10^{-1}$ \\ \hline
        %\rowcolor{lightgray}
       %\rowcolor{lightgray} \multicolumn{6}{|c|}{\textbf{$\mathcal{M}$}}\\
        %\rowcolor{lightgray}
        %\rowcolor{lightgray}\multicolumn{6}{|c|}{\textbf{$\alpha =1\times 10^{-1}$ , $\beta = 1$}}\\
        %\hline
        %$Iters$ & 109 & 80 & 65  & 71 & 78 \\
        %$CPU$ & 0.15 & 0.25 & 1.17  & 6.66 & 36.99 \\
        %$R_k$ &  $7.0\times10^{-7}$ &  $8.4\times10^{-7}$ &  $7.1\times10^{-7}$   & $9.0\times10^{-7}$  & $1.0\times10^{-6}$ \\
        %$E_k$ &  $3.0\times10^{-7}$ &  $6.0\times10^{-7}$ &  $1.5\times10^{-6}$  & $8.3\times10^{-6}$ & $1.4\times10^{-6}$ \\ \hline
        %\rowcolor{lightgray}
        \rowcolor{lightgray}\multicolumn{6}{|c|}{\textbf{$\mathcal{M}$}}\\
        % \rowcolor{lightgray}
        %\rowcolor{lightgray}\multicolumn{6}{|c|}{\textbf{$\alpha =10^{-3}$ , $\beta = 1$}}\\
        \hline
        Iters & 109 & 75 & 54  & 60 & 74  \\
        Prec.CPU & 0.0016 & 0.0037 & 0.0389  & 0.2334  & 1.8452 \\
        CPU &  0.05 &  0.08 &  0.24 & 1.67 &  14.08 \\
        Total.CPU &  0.05 & 0.08 & 0.28  & 1.91 & 15.92 \\
        $R_k$ & $7.3\times10^{-7}$ & $9.9\times10^{-7}$ & $7.8\times10^{-7}$ & $8.8\times10^{-7}$  & $6.0\times10^{-7}$ \\
        $E_k$ & $2.9\times10^{-7}$ & $9.3\times10^{-7}$ & $2.5\times10^{-6}$ & $5.5\times10^{-6}$ & $7.8\times10^{-6}$ \\
        \hline
        %\rowcolor{lightgray}
       \rowcolor{lightgray} \multicolumn{6}{|c|}{\textbf{$\mathcal{P}_{2}$}}\\ \hline
        Iters & 115 & 402 &  $\ddag$ & $\ddag$  & - \\
        Prec.CPU & 0.0286 & 0.2170 & 6.0734  & 261.1566  & - \\
        CPU & 0.05  & 1.77  & 42.95 & 454.22 & $\dagger$  \\
        Total.CPU &  0.08 & 1.98 &  49.02 & 715.38 & - \\
        $R_k$ & $8.4\times10^{-7}$ & $9.5\times10^{-7}$ & $1.6\times10^{-3}$ & $7.0\times10^{-1}$  & -\\
        $E_k$ & $1.3\times10^{-6}$ & $7.9\times10^{-6}$ & $2.2\times10^{-4}$ & $1.7\times10^{0}$ & - \\ \hline
        %\rowcolor{lightgray}
       \rowcolor{lightgray} \multicolumn{6}{|c|}{\textbf{$\mathcal{P}_{BD2}$}}\\ \hline
        Iters & 156 & 569 & $\ddag$  & $\ddag$ & -  \\
        Prec.CPU & 0.0286 & 0.2170 & 6.0734  & 261.1566 & - \\
        CPU &  0.05 & 3.15 & 41.72 & 451.88 & $\dagger$  \\
        Total.CPU & 0.08 & 3.37 & 47.80 & 713.04 & -  \\
        $R_k$ & $8.7\times10^{-7}$ & $8.8\times10^{-7}$ & $1.9\times10^{-3}$ & $4.9\times10^{-1}$ & - \\
        $E_k$ & $1.1\times10^{-6}$ & $1.3\times10^{-5}$ & $5.5\times10^{-4}$ & $4.7\times10^{-1}$ & - \\ \hline
    \end{tabular}}
\end{table}}

%==================================EXAMPLE 2 ================================================================
\begin{ex}\label{ex2}\cite{Ex1-1}
    Consider the saddle point problem \eqref{Sdpr1}, for which
    \[
    A=\text{diag}(2W^TW+D_1,D_2,D_3)\in \mathbb{R}^{n\times n}
    \]
    is a block-diagonal matrix,
    \[
    B=[E,-I_{2\widetilde{p}},I_{2\widetilde{p}}]\in \mathbb{R}^{m\times n} \quad \text{and}\quad C=E^T \in \mathbb{R}^{\ell \times m}
    \]
    are both full row-rank matrices, where $\widetilde{p}=p^2$, $\widehat{p}=p(p+1)$; $W=(w_{ij}) \in \mathbb{R}^{\widehat{p}\times \widehat{p}}$ with $w_{ij} =e^{-2((i/3)^2+(j/3)^2)}$; $D_1=I_{\widehat{p}}$ is an identity matrix; $D_i=\text{diag}(d_j ^{(i)})\in \mathbb{R}^{2\widetilde{p}\times 2\widetilde{p}}$, $i=2,3,$ are diagonal matrix with
    \begin{equation*}
        d_j ^{(2)}=\begin{cases}
            1, & \text{for} \quad 1\leq j\leq\widetilde{p},\\
            10^{-5}(j-\widetilde{p})^2, & \text{for} \quad \widetilde{p}+1\leq j \leq 2\widetilde{p},
        \end{cases}
    \end{equation*}
    \[
    d_j ^{(3)}=10^{-5}(j+\widetilde{p})^2, \qquad \text{for} \quad 1\leq j\leq 2\widetilde{p},
    \]
    and
    \[
    E=\begin{pmatrix}
        \widehat{E}\otimes I_p \\
        I_p \otimes \widehat{E}
    \end{pmatrix},
    \quad \widehat{E}=\begin{pmatrix}
        2 & -1 & & & \\
        &  2 & 1 & & \\
        & & \ddots & \ddots & \\
        & & & 2 & -1
    \end{pmatrix} \in  \mathbb{R}^{p \times (p+1)}.
    \]

Similar to Example \ref{ex1}, the Matrix $\hat{S}$ is set to be  $\hat{S}=B{diag(A)}^{-1} B^T$. In addition, we choose the parameters $\alpha = 10^{-1}$ and $\beta = 1$.  Numerical results  for different values of $p$ are listed in Table \ref{tab3}. As we observe, the numerical results illustrate that the preconditioner $\mathcal{M}$  considerably reduces the CPU time of the FGMRES method without preconditioning.  

We have also applied the Full-GMRES iteration method incorporated with preconditioners $\mathcal{M}$, $\mathcal{P}_{2}$ and $\mathcal{P}_{BD2}$. In this case, the subsystems were solved exactly using the Cholesky factorization. Our numerical results show that, the 
Full-GMRES method outperforms the Full-GMRES method in conjunction with the preconditioners (all the three preconditioners) from the elapsed CPU time point of view for large values of $p$.  Hence, we have not reported the numerical results.
\end{ex}
%==================================Table 3 ================================================================
\begin{table}
    \tiny{
    \caption{Numerical results of FGMRES for  Example \ref{ex2}.}\label{tab3}
    \centering
    \begin{tabular}{|c||c|c|c|c|c|c|}\hline
        $p$ & 16 & 32 & 64 & 128 & 256 &512 \\ \hline
        \rowcolor{lightgray}
        \multicolumn{7}{|c|}{\textbf{No Preconditioning}}\\ \hline
        Iters & 186 & 190 & 187 & 180 & 172 & 163 \\
        CPU & 0.08 & 0.37 & 0.97 & 2.86 & 17.78 &  96.70\\
        $R_k$ & $1.0\times10^{-6}$ & $9.9\times10^{-7}$ & $1.0\times10^{-6}$ &  $9.8\times10^{-7}$ &  $9.6\times10^{-7}$ & $1.0\times10^{-6}$\\
        $E_k$ & $1.3\times10^{-5}$ & $1.4\times10^{-5}$ & $1.4\times 10^{-5}$ &  $1.4\times10^{-5}$ &  $1.4\times10^{-5}$ &  $1.4\times10^{-5}$\\ \hline
        \rowcolor{lightgray}
        \multicolumn{7}{|c|}{\textbf{$\mathcal{M}$}}\\
%        \rowcolor{lightgray}\multicolumn{7}{|c|}{\textbf{$\alpha =10^{-1}$ , $\beta = 1$}}\\
        \hline
        Iters & 70 & 69 & 68 & 65 & 63 & 60\\
        CPU & 0.09 & 0.29 & 0.85 & 2.18 & 10.30 & 45.97 \\
        $R_k$ &  $1.0\times10^{-6}$ &  $9.5\times10^{-7}$ &  $8.8\times10^{-7}$ &  $9.3\times10^{-7}$ &  $8.5\times10^{-7}$ &  $9.5\times10^{-7}$\\
        $E_k$ &  $5.6\times10^{-6}$ &  $5.7\times10^{-6}$ &  $5.0\times10^{-6}$ &  $5.5\times10^{-6}$ &  $4.9\times10^{-6}$  &  $5.2\times10^{-6}$\\ \hline
        \rowcolor{lightgray}
        \multicolumn{7}{|c|}{\textbf{$\mathcal{P}_{2}$}}\\ \hline
        Iters & 13 & 13 & 13  &13   & $-$ & $-$ \\
        Prec.CPU & 0.0015 & 0.01 & 0.37  &13.75   & $-$ & $-$ \\
        CPU & 0.02 & 0.18 & 2.59  &34.55   & $\dagger$ &$\dagger$ \\
        Total.CPU & 0.03 & 0.20 & 2.96  &48.31   &$-$ &  $-$\\
        $R_k$ &  $5.1\times10^{-7}$ & $3.3\times10^{-7}$ & $3.4\times10^{-7}$  & $2.5\times10^{-7}$   & $-$ & $-$\\
        $E_k$ & $2.1\times10^{-6}$ & $1.5\times10^{-6}$ &  $1.1\times10^{-6}$ &$8.0\times10^{-7}$   & $-$ & $-$  \\ \hline
        \rowcolor{lightgray}
        \multicolumn{7}{|c|}{\textbf{$\mathcal{P}_{BD2}$}}\\ \hline
        Iters & 19 & 18 & 18  &18   & $-$  & $-$ \\
        Prec.CPU & 0.0015 & 0.01 & 0.37  &13.75  & $-$ & $-$ \\
        CPU & 0.03 & 0.25 & 3.30  & 49.54  & $\dagger$ & $\dagger$ \\
        Total.CPU & 0.03 & 0.26 & 3.67  &63.29   & $-$ & $-$ \\
        $R_k$ &  $1.2\times10^{-7}$ & $9.6\times10^{-7}$ & $8.4\times10^{-7}$  &$9.8\times10^{-7}$   & $-$ & $-$ \\
        $E_k$ & $6.1\times10^{-7}$ & $1.9\times10^{-6}$ &  $1.8\times10^{-6}$ &$3.2\times10^{-6}$   & $-$ & $-$ \\ \hline
    \end{tabular}}
\end{table}

%==================================Table 4 ================================================================

\section{Conclusion} \label{sec4}
A new block diagonal preconditioner has been presented for a class of $3\times3$ block saddle point problems. This preconditioner is based on augmentation and performs well in practice. Also, it is  easy to implement and has much better efficiency than the recently existing preconditioners. We  have further estimated the lower and upper bounds of eigenvalues of the preconditioned matrix. We have examined the new preconditioner to accelerate the convergence speed of the FGMRES method as well as Full-GMRES. Our numerical tests illustrate that the proposed preconditioner is quite suitable and is superior to the other tested preconditioners in the literature.

\section*{Acknowledgments}
The authors would like to thank the anonymous referees for their useful comments and suggestions.

%\pagebreak

\end{document}